\documentclass{article}%
\usepackage{amsfonts}
\usepackage{amsmath}
\usepackage{amssymb}
\usepackage{graphicx}
\usepackage[pdfstartview=FitH]{hyperref}%
\providecommand{\U}[1]{\protect\rule{.1in}{.1in}}
\newtheorem{theorem}{Theorem}[section]

\newtheorem{conjecture}[theorem]{Conjecture}
\newtheorem{corollary}[theorem]{Corollary}
\newtheorem{definition}{Definition}
\newtheorem{example}[theorem]{Example}
\newtheorem{lemma}[theorem]{Lemma}

\newtheorem{proposition}[theorem]{Proposition}
\newtheorem{remark}[theorem]{Remark}
\newenvironment{proof}[1][Proof]{\noindent\textbf{#1.} }{\ \rule{0.5em}{0.5em}}

\begin{document}

\title{Application of vertex algebras to the structure theory of certain
representations over the Virasoro algebra}
\author{Gordan Radobolja\\Faculty of Natural sciences and Mathematics,\\University of Split, Croatia}
\maketitle

\begin{abstract}
In this paper we discuss the structure of the tensor product $V_{\alpha,\beta
}^{\prime}\otimes L(c,h)$ of irreducible module from intermediate series and
irreducible highest weight module over the Virasoro algebra. We generalize
Zhang's irreducibility criterion from \cite{Zhang}, and show that
irreducibility depends on the existence of integral roots of a certain
polynomial, induced by a singular vector in the Verma module $V(c,h)$. A new
type of irreducible $\operatorname*{Vir}$-module with infinite-dimensional
weight subspaces is found. We show how the existence of intertwining operator
for modules over vertex operator algebra yields reducibility of $V_{\alpha
,\beta}^{\prime}\otimes L(c,h)$ which is a completely new point of view to
this problem. As an example, the complete structure of the tensor product with
minimal models $c=-22/5$ and $c=1/2$ is presented.\medskip

\textbf{Keywords}: Virasoro algebra, highest weight module, intermediate
series, minimal model, vertex operator algebra, intertwining operator

\textbf{AMS classification}: 17B10, 17B65, 1768, 1769

\end{abstract}

\section{Introduction}

\label{intro}

Along with affine Kac-Moody algebras, the Virasoro algebra plays an important
role in the theory of infinite-dimensional Lie algebras. Its irreducible
weight modules with finite-dimensional weight spaces were classified in
\cite{Mathieu} - every such module is either highest (or lowest) weight
module, or module belonging to intermediate series. It has been shown in
\cite{Mazorchuk-Zhao1} that irreducible weight module cannot have
finite-dimensional and infinite-dimensional weight subspaces simultaneously.
Irreducible modules with infinite-dimensional weight subspaces have been
studied recently by several authors. The earliest examples were constructed by
H.\ Zhang in \cite{Zhang}. Later, many new classes were given in
\cite{Conley-Martin} and \cite{Liu-Lu-Zhao}. All these modules are included in
a huge family or irreducible weight modules constructed in \cite{Lu-Zhao}.
Also, various families of nonweight irreducible modules were studied at the
same time. Most examples are various versions of Whittaker modules, and are
described in a uniform way in \cite{Mazorchuk-Zhao2}.

Modules with infinite-dimensional weight subspaces over some other Lie
algebras were studied motivated by their connection with the theory of vertex
operator algebras (VOAs) and fusion rules in conformal field theory. Such
modules over affine Lie algebras were constructed in \cite{Chari-Pressley},
while their relation to VOA theory was discussed in \cite{Adamovic1},
\cite{Adamovic2}, and \cite{Adamovic3}.

One series of such representations over the Virasoro algebra is the tensor
product $V_{\alpha,\beta}^{\prime}\otimes L(c,h)$ of irreducible module
$V_{\alpha,\beta}^{\prime}$ from intermediate series and irreducible highest
weight module $L(c,h)$ introduced in \cite{Zhang}. However, this module is not
always irreducible. Zhang has given irreducibility criterion under the
condition $\alpha\notin\beta\mathbb{Z}+\mathbb{Z}$, but this rules out many
interesting examples appearing in the theory of VOA and intertwining operators.

In this paper we generalize Zhang's criterion (Theorem \ref{main}) and use it
to obtain new series of irreducible modules. Basically, $V_{\alpha,\beta
}^{\prime}\otimes L(c,h)$ is irreducible if and only if the equations
$x(v_{n+1}\otimes v)=v_{n}\otimes v$ have solutions in the universal
enveloping algebra $U(\operatorname*{Vir})$ for every $n\in\mathbb{Z}$. Such a
solution can exists only if there is a singular vector in the Verma module
$V(c,h)$. We show that irreducibility of tensor product module depends on
(non)existence of integral roots of a certain polynomial induced by that
singular vector (Theorem \ref{poly}).

We also discuss the structure of a reducible module $V_{0,\beta}^{\prime
}\otimes L(c,0)$ (Theorem \ref{n}) and present a new type of irreducible
$\operatorname*{Vir}$-module with infinite-dimensional weight subspaces
(Theorem \ref{m}).

In section \ref{sec} we show strong connection between reducibility of
$V_{\alpha,\beta}^{\prime}\otimes L(c,h)$ and the existence of certain
intertwining operators for vertex operator algebra $L(c,0)$. In subsection
\ref{min} we focus on the so called minimal models. $L(c_{p,q},0)$ is a
rational VOA, and all of its irreducible modules are known, together with
fusion rules for intertwining operators between them. Using these fusion rules
we prove reducibility of certain modules $V_{\alpha,\beta}^{\prime}\otimes
L(c,h)$. Combining these two methods, i.e.\ fusion rules to show reducibility
and singular vectors to prove irreducibility, we give complete results for
$V_{\alpha,\beta}^{\prime}\otimes L(c,h)$ when $L(c,h)$ is a minimal model for
$c=c_{2,5}=-22/5$, and $c=c_{3,4}=1/2$. When these tensor product modules are
reducible, we obtain irreducible minimal models as quotient modules
(Propositions \ref{p1}-\ref{p3}). Based on these examples we conjecture that
analogous results hold in general, for all minimal models. In subsection
\ref{c=1} we focus on $c=1$ and demonstrate that reducibility of
$V_{\alpha,\beta}^{\prime}\otimes L(1,h)$ again coincides with the existence
of known intertwining operators for $h=m^{2}$ and $h=m^{2}/4$.

\textbf{Remark }The irreducibility problem for the tensor product has been
investigated in \cite{Chen-Guo-Zhao} at the same time as in this paper. Making
use of a "shifting technique", these authors have also shown that reducibility
coincides with the existence of integral roots of associated polynomial. Their
elegant proof works for minimal models as well as in general case. However,
our approach using fusion rules and VOA gives a better understanding of a
subquotient structure.

\section{Preliminaries}

The Virasoro algebra $\operatorname*{Vir}$ is a complex Lie algebra spanned by
$\{C,L_{i}:i\in\mathbb{Z}\}$ where $C$ is a central element, and $[L_{n}%
,L_{m}]=\left(  n-m\right)  L_{m+n}+\delta_{m,-n}\frac{n^{3}-n}{12}C$. It has
a natural triangular decomposition
\[
\operatorname*{Vir}\nolimits_{+}=\bigoplus\limits_{n>0}\mathbb{C}L_{n}%
\quad\operatorname*{Vir}\nolimits_{-}=\bigoplus\limits_{n>0}\mathbb{C}%
L_{-n}\quad\operatorname*{Vir}\nolimits_{0}=\mathbb{C}L_{0}\oplus\mathbb{C}C
\]
Let $U(\operatorname*{Vir})$ and $U(\operatorname*{Vir}_{-})$ denote the
universal enveloping algebras of $\operatorname*{Vir}$ and
$\operatorname*{Vir}_{-}$, respectively. The vectors $L_{i_{1}}\cdots
L_{i_{n}}\mathbf{1}$, $i_{1}\leq\cdots\leq i_{n}$, $n\in\mathbb{N}_{0}$ form a
standard PBW\ basis of $U(\operatorname*{Vir})$. If we let $i_{n}<0$, we
obtain a basis for $U(\operatorname*{Vir}_{-})$.

In the following we recall the well known facts from representation theory of
the Virasoro algebra.

Verma module $V(c,h)$ with highest weight $h$ and central charge $c$ is a
$\operatorname*{Vir}$-module generated by the so called highest weight vector
$v$ such that $Cv=cv$, $L_{0}v=hv$ and $L_{n}v=0$ for $n>0$. It is a free
$U(\operatorname*{Vir}_{-})$-module with PBW\ basis $\{L_{-i_{n}}\cdots
L_{-i_{1}}v:i_{n}\geq\cdots\geq i_{1}>0,n\in\mathbb{N}_{0}\mathbb{\}}$.
Nonzero vector $u\in V(c,h)$ is called a singular vector if
$\operatorname*{Vir}_{+}u=0$ and $L_{0}u=(h+n)u$ for some $n\in\mathbb{N}_{0}%
$. Each singular vector generates a submodule of $V(c,h)$, and every proper
nontrivial submodule contains at least one singular vector. Every Verma module
contains proper maximal (possibly trivial) submodule $J(c,0)$, and the
quotient $L(c,h)=V(c,h)/J(c,h)$ is a unique irreducible highest weight module
with the highest weight $(c,h)$. Each highest weight module is a quotient of
the corresponding Verma module. If $J(c,h)$ is nontrivial, it is generated
either by one or by two singular vectors of different weights. In this paper,
we use the notation $c_{p,q}=1-6\frac{(p-q)^{2}}{pq}$ for $p,q>1$ relatively
prime. $J(c,h)$ is generated by two singular vectors only when $c=c_{p,q}$. We
say that $V(c,h)$ is reducible with degree $m$ if $m$ is the smallest positive
integer such that $J(c,h)$ contains a singular vector of weight $h+m$.

Let $\alpha,\beta\in\mathbb{C}$. $\operatorname*{Vir}$-module from
intermediate series is defined by $V_{\alpha,\beta}=\sum_{m\in\mathbb{Z}%
}\mathbb{C}v_{m}$ with $L_{n}v_{m}=-(m+\alpha+\beta+n\beta)v_{m+n}$ and
$Cv_{m}=0$. Module $V_{\alpha,\beta}$ is reducible if and only if $\alpha
\in\mathbb{Z}$ and $\beta=0$, or $\beta=1$. Since $V_{\alpha,\beta}\cong
V_{\alpha+k,\beta}$ for all $k\in\mathbb{Z}$, we may assume $\alpha=0$ if
$\alpha\in\mathbb{Z}$. Define $V_{0,0}^{\prime}:=V_{0,0}/\mathbb{C}v_{0}$,
$V_{0,1}^{\prime}:=\sum_{m\neq-1}\mathbb{C}v_{m}$, and $V_{\alpha,\beta
}^{\prime}:=V_{\alpha,\beta}$ for all other pairs $(\alpha,\beta)$. Then
$V_{\alpha,\beta}^{\prime}$ are all irreducible modules belonging to the
intermediate series.

\bigskip

We define a module structure on tensor products $V_{\alpha,\beta}^{\prime
}\otimes L(c,h)$ by
\[
L_{n}(v_{k}\otimes x)=(L_{n}v_{k})\otimes x+v_{k}\otimes(L_{n}x).
\]
It is easy to see that
\[
\left(  V_{\alpha,\beta}^{\prime}\otimes L\left(  c,h\right)  \right)
_{h-\alpha-\beta+m}=\bigoplus\limits_{n\in\mathbb{Z}_{+}}\mathbb{C}%
v_{n-m}\otimes L\left(  c,h\right)  _{h+n}
\]
so $V_{\alpha,\beta}^{\prime}\otimes L(c,h)$ has infinite-dimensional weight
subspaces. Also, $V_{\alpha,\beta}^{\prime}\otimes L(c,h)$ is generated by
$\{v_{m}\otimes v:m\in\mathbb{Z}\}$ where $v$ is the highest weight vector in
$L(c,h)$.

\section{Irreducibility of modules $V_{\alpha,\beta}^{\prime}\otimes L(c,h)$}

The following irreducibility criterion was proved in \cite{Zhang}:

\begin{theorem}
[\cite{Zhang}]If $\alpha\notin\beta\mathbb{Z}+\mathbb{Z}$, then $V_{\alpha
,\beta}^{\prime}\otimes L(c,h)$ is irreducible if and only if $V_{\alpha
,\beta}^{\prime}\otimes L(c,h)$ is cyclic on every $v_{m}\otimes v$,
$m\in\mathbb{Z}$, where $v$ is the highest weight vector of $L(c,h).$
\end{theorem}

Here we expand Zhang's proof and eliminate restriction on $\alpha$. (See also
Theorem 1 of \cite{Chen-Guo-Zhao} for a different proof of this theorem.)

\begin{theorem}
\label{main}Module $V_{\alpha,\beta}^{\prime}\otimes L(c,h)$ is irreducible if
and only if $V_{\alpha,\beta}^{\prime}\otimes L(c,h)$ is cyclic on every
$v_{m}\otimes v$, $m\in\mathbb{Z}$, where $v$ is the highest weight vector of
$L(c,h).$
\end{theorem}

\begin{proof}
The only if part is trivial. Suppose $V=V_{\alpha,\beta}^{\prime}\otimes
L(c,h)$ is cyclic on every $v_{m}\otimes v$. Let $U$ be a nontrivial submodule
in $V$, and let $x\in U$ be a nonzero weight $n-m$ vector:
\[
x=v_{m-n}\otimes x_{0}+v_{m-n+1}\otimes x_{1}+\cdots+v_{m}\otimes x_{n}
\]
where
\[
x_{j}\in V(c,h)_{h+j},j=0,1,\ldots,n\text{ and }x_{n}\neq0.
\]
Using induction on $n$ we find $v_{k}\otimes v\in U$ which proves irreducibility.

If $n=0$, then $x=v_{m}\otimes x_{0}\in U$ is a multiple of $v_{m}\otimes v$.
Assume $n>0$. Since $L_{1}$ and $L_{2}$ generate the algebra
$\operatorname*{Vir}_{+}$ there exists $i\in\left\{  1,2\right\}  $ such that
$L_{i}x_{n}\neq0$ (otherwise, $x_{n}$ would be a highest weight vector, a
contradiction with irreducibility of $L(c,h)$). The idea is to choose $w\in
U\left(  \operatorname*{Vir}_{+}\right)  $ such that $wx\in U$ has at most $n$
components. By inductive hypothesis there is $v_{k}\otimes v\in U$ for some
$k\in\mathbb{Z}$, as long as $wx\neq0$. Then we find a component of $wx$ in
$\mathbb{C}v_{m+l-i}\otimes L(c,h)_{h+n-i}$ to check weather $wx=0$ and repeat
the process if necessary.

For any $l\in\mathbb{Z}$ such that $l>n+i$, $m+\alpha+\beta+(l-i)\beta$ and
$m+\alpha+\beta+l\beta$ are nonzero let
\[
w=\left(  m+\alpha+\beta+i\beta\right)  \left(  m+i+\alpha+\beta+\left(
l-i\right)  \beta\right)  L_{l}+\left(  m+\alpha+\beta+l\beta\right)
L_{l-i}L_{i}.
\]
It is easy to show that $wv_{m}=0$. Since $l>n+i$ we have $L_{l}x_{j}%
=L_{l-i}x_{j}=0$ for $j=0,1,\ldots,n$, so
\begin{gather}
wx=w\sum_{j=0}^{n}v_{m-j}\otimes x_{n-j}=\label{h}\\
=\sum_{j=1}^{n}wv_{m-j}\otimes x_{n-j}+\left(  m+\alpha+\beta+l\beta\right)
\sum_{j=0}^{n-i}L_{l-i}v_{m-j}\otimes L_{i}x_{n-j}.\nonumber
\end{gather}
Therefore, if $wx\neq0$ by the induction hypothesis there exists $v_{k}\otimes
v\in U$. Suppose $wx=0$. The component of $wx$ in $\mathbb{C}v_{m+l-i}\otimes
L(c,h)_{h+n-i}$ is
\[
X_{l}^{i}=wv_{m-i}\otimes x_{n-i}+\left(  m+\alpha+\beta+l\beta\right)
L_{l-i}v_{m}\otimes L_{i}x_{n},
\]
and we have
\[
\left(  m+\alpha+\beta+l\beta\right)  L_{l-i}v_{m}\otimes L_{i}x_{n}\neq0.
\]
If $x_{n-i}$ and $L_{i}x_{n}$ are linearly independent or $x_{n-i}=0$, then
$X_{l}^{i}\neq0$, so $wx\neq0$. (In case $\alpha=0$, $\beta=1$ and $m-i=-1$ we
have $X_{l}^{i}=\left(  m+l+1\right)  L_{l-i}v_{m}\otimes L_{i}x_{n}\neq0$,
and if $\alpha=\beta=m-i=0$ we get $X_{l}^{i}=mL_{l-i}v_{m}\otimes L_{i}%
x_{n}\neq0$.)

Suppose $0\neq\lambda x_{n-i}=L_{i}x_{n}$ for some $\lambda\in\mathbb{C}%
^{\ast}$. If $X_{l}^{i}=0$ we obtain
\begin{align}
&  -\left(  m+\alpha+\beta+i\beta\right)  \left(  m+i+\alpha+\beta+\left(
l-i\right)  \beta\right)  \left(  m-i+\alpha+\beta+l\beta\right) \nonumber\\
&  +\left(  m-i+\alpha+\beta+i\beta\right)  \left(  m+\alpha+\beta+\left(
l-i\right)  \beta\right)  \left(  m+\alpha+\beta+l\beta\right)  +\label{l}\\
&  -\lambda\left(  m+\alpha+\beta+\left(  l-i\right)  \beta\right)  \left(
m+\alpha+\beta+l\beta\right)  =0\nonumber
\end{align}
for every $l>n-i$. Therefore, equation (\ref{l}) holds for every
$l\in\mathbb{C}$ and, in particular for $l=-\beta^{-1}(m+\alpha+\beta)$ if
$\beta\neq0$. Then
\[
i^{2}\left(  m+\alpha+\beta+i\beta\right)  \left(  1-\beta\right)  =0.
\]
Since $i\neq0$ we get $\beta=1$ or $m+\alpha+\beta+i\beta=0$.

\begin{description}
\item[\fbox{$\beta=1$}] Suppose that $\beta=1$. Then from (\ref{l}) it
follows\footnote{In the original proof, the author makes a wrong conclusion
that $\lambda=i=m+\alpha+1.$ We expand this part of the proof.} that
$\lambda=-i$.

\item For $l>n+1+i$ and $m+\alpha+l+1\neq i$ set
\[
w_{1}=\left(  m+\alpha+i+2\right)  L_{l}+L_{l-i-1}L_{i+1}.
\]
Then $w_{1}v_{m}=0$ and
\[
w_{1}x=\sum_{j=1}^{n}w_{1}v_{m-j}\otimes x_{n-j}+\sum_{j=0}^{n-i-1}%
L_{l-i-1}v_{m-j}\otimes L_{i+1}x_{n-j}.
\]
Again, if $w_{1}x\neq0$ we can find $v_{k}\otimes v\in U$ by induction. The
component of $w_{1}x$ in $\mathbb{C}v_{m+l-i}\otimes L(c,h)_{h+n-i}$ is
\[
-i\left(  m+1+\alpha+l-i\right)  v_{m+l-i}\otimes x_{n-i}
\]
and since $x_{n-i}\neq0$ we obtain $w_{1}x\neq0$.

\item[\fbox{$m+\alpha+\beta+i\beta=0\neq\beta$}] Now assume $m+\alpha
+\beta+i\beta=0\neq\beta$ (This case is not covered in Zhang's proof due to
the condition $\alpha\notin\beta\mathbb{Z}+\mathbb{Z}$.) Then $L_{i}v_{m}=0$,
hence
\[
L_{i}x=\sum\limits_{j=1}^{n}L_{i}v_{m-j}\otimes x_{n-j}+\sum_{j=0}%
^{n-i}v_{m-j}\otimes L_{i}x_{n-j}.
\]
If $L_{i}x\neq0$ we can find $v_{k}\otimes v\in U$ as in previous cases. Let
$L_{i}x=0$. Consider the component of $L_{i}x$ in $\mathbb{C}v_{m}\otimes
L\left(  c,h\right)  _{h+n-i}$:%
\[
iv_{m}\otimes x_{n-i}+v_{m}\otimes L_{i}x_{n}=0.
\]
Therefore, $L_{i}x_{n}=-ix_{n-1}\neq0$. Set
\[
w_{2}=\left(  \left(  l-2i-1\right)  \beta+i+1\right)  L_{l}+\left(
l-i\right)  L_{l-i-1}L_{i+1}.
\]
Again, since $w_{2}v_{m}=0$ we get
\[
w_{2}x=\sum_{j=1}^{n}w_{2}v_{m-j}\otimes x_{n-j}+\left(  l-i\right)
\sum_{j=0}^{n-i-1}L_{l-i-1}v_{m-j}\otimes L_{i+1}x_{n-j}.
\]
If $w_{2}x\neq0$ our proof is done. Now suppose $w_{2}x=0$ and consider its
component in $\mathbb{C}v_{m-i}\otimes x_{n-j}$. For $l\in\mathbb{C}$ we have%
\[
-\left(  \left(  l-2i-1\right)  \beta+i+1\right)  \left(  \left(  l-i\right)
\beta-i\right)  +\left(  l-i\right)  \left(  \left(  l-2i-1\right)
\beta-1\right)  =0
\]
In particular, for $l=i$ we have
\[
i\left(  i+1\right)  \left(  1-\beta\right)  =0,
\]
so $\beta=1$. This case is already covered.

\item[\fbox{$\beta=0$}] Now let $\beta=0$. Then (\ref{l}) becomes
\[
\left(  m+\alpha\right)  \left(  i\left(  m-i+\alpha\right)  +\lambda\left(
m+\alpha\right)  \right)  =0.
\]
If $\alpha=-m\in\mathbb{Z}$ we may assume $\alpha=0$ but then $m=0,$ so
$v_{0}\in V_{0,0}^{\prime}$ which contradicts the definition of $V_{0,0}%
^{\prime}$. Therefore $\lambda=-i\frac{m+\alpha-i}{m+\alpha}$ i.e.\
\[
L_{i}x_{n}=-i\frac{m+\alpha-i}{m+\alpha}x_{n-i}.
\]

\item For $l>n+2$, $m+l+\alpha\neq0$ define
\[
w_{3}=\left(  m+i+\alpha\right)  \left(  m+l+\alpha\right)  L_{2l}%
-L_{l}L_{l-i}L_{i}.
\]
Then $w_{3}v_{m}=0$ and if $w_{3}x\neq0$, our proof is done. The component of
$w_{3}x$ in $\mathbb{C}v_{m+2l-i}\otimes L(c,h)_{h+n-i}$ is
\[
w_{3}v_{m-i}\otimes x_{n-i}-L_{l}L_{l-i}v_{m}\otimes L_{i}x_{n},
\]
so $w_{3}x=0$ yields
\[
i\left(  m-i+\alpha\right)  \left(  m+i+\alpha\right)  =0.
\]
Therefore $\alpha\in\mathbb{Z}$ and we may assume $\alpha=0$. Then $m=-i$
since $\lambda\neq0$. Let $w_{4}=L_{l}+L_{l-i-1}L_{i+1}$. Then $w_{4}v_{-i}=0
$,
\[
w_{4}x=\sum_{i=1}^{n}wv_{-i-j}\otimes x_{n-j}+\sum_{j=0}^{n-i-1}%
v_{-i-j}\otimes L_{i+1}x_{n-j}
\]
and the component of $w_{4}x$ in $\mathbb{C}v_{l-2i}\otimes L\left(
c,h\right)  _{h+n-i}$ is $2i^{2}v_{l-2i}\otimes x_{n-i}\neq0$. Once again we
conclude that $w_{4}x\neq0$ and by induction we find that $v_{k}\otimes x\in
U$.
\end{description}

This completes the proof.
\end{proof}

\bigskip

From this point on, we write $U_{n}:=U(\operatorname*{Vir})(v_{n}\otimes v)$.
We note that from the proof of the previous theorem we have

\begin{corollary}
\label{f1}Let $M$ be a nontrivial submodule in $V_{\alpha,\beta}^{\prime
}\otimes L(c,h)$. Then $M$ contains $U_{n}$ for some $n\in\mathbb{Z}$.
\end{corollary}

In order to prove irreducibility of $V_{\alpha,\beta}^{\prime}\otimes L(c,h)$
it suffices to check that $U_{n}=U_{n+1}$ for every $n\in\mathbb{Z}$. Since
\[
L_{1}(v_{n}\otimes v)=-(n+\alpha+2\beta)v_{n+1}\otimes v
\]
we get $U_{n}\supseteq U_{n+1}$ if $\alpha+2\beta\notin\mathbb{Z}$. The
opposite inclusion, however, requires a certain relation to hold in $L(c,h)$
i.e.\ the existence of a singular vector in $V(c,h).$ For example,
$V_{\alpha,\beta}^{\prime}\otimes V(c,h)$ is always reducible, even if
$V(c,h)$ is irreducible, because $U_{n}\nsubseteq U_{n+1}$ for every $n$
(Theorem 3 in \cite{Zhang}). If, on the other hand, Verma module $V(c,h)$ is
reducible with degree $m$, then $V_{\alpha,\beta}^{\prime}\otimes L(c,h)$ is
irreducible whenever $\alpha$ is either transcendental over $\mathbb{Q}%
(c,h,\beta)$, or algebraic over $\mathbb{Q}(c,h,\beta)$ with degree greater
then $m$ (Theorem 5 in \cite{Zhang}). Basically, if there is a weight $m$
singular vector in $V(c,h)$, we can prove irreducibility of $V_{\alpha,\beta
}^{\prime}\otimes L(c,h)$ for all pairs $(\alpha,\beta)\in\mathbb{C}^{2}$
except for those forming a certain algebraic curve (integral roots of a degree
$m$ polynomial). In the next section we expand this result in more details,
and prove its converse.

\section{Subquotient structure}

First we show that highest weight modules appear naturaly as subquotients in
tensor products.

\begin{lemma}
\label{kvoc}Let $M(c,h)$ be a highest weight module. If $U_{n+k}\subsetneq
U_{n}\subseteq V_{\alpha,\beta}^{\prime}\otimes M(c,h)$ for every
$k\in\mathbb{N}$, then $U_{n}/U_{n+1}$ is a highest weight module with the
highest weight $h-\alpha-\beta-n$.
\end{lemma}

\begin{proof}
Since
\begin{align*}
L_{k}(v_{n}\otimes v)  &  \in U_{n+k}\subseteq U_{n+1},\text{\quad}%
k\in\mathbb{N}\\
L_{0}(v_{n}\otimes v)  &  =-(\alpha+\beta)v_{n}\otimes v+h(v_{n}\otimes v)
\end{align*}
we conclude that $v_{n}\otimes v+U_{n+1}$ is a cyclic highest weight vector
with the highest weight $h-\alpha-\beta-n$.
\end{proof}

The following important result generalizes the corollary following Theorem 4
in \cite{Zhang}

\begin{proposition}
\label{j}If $\alpha\notin\mathbb{Z}$, module $V_{\alpha,\beta}^{\prime}\otimes
L(c,0)$ is irreducible.
\end{proposition}

\begin{proof}
Since $L_{-1}v=0$ in $L(c,0)$, we have
\begin{align*}
L_{-1}(v_{n}\otimes v)  &  =-(n+\alpha)(v_{n-1}\otimes v)\text{, }\forall
n\in\mathbb{Z},\\
L_{1}(v_{n}\otimes v)  &  =-(n+\alpha+2\beta)(v_{n+1}\otimes v)\text{,
}\forall n\in\mathbb{Z}.
\end{align*}
If $n+\alpha+2\beta\in\mathbb{Z}$ we use
\[
L_{-1}L_{2}\left(  v_{n}\otimes v\right)  =\beta(n+2+\alpha)(v_{n+1}\otimes
v)\neq0
\]
to prove $U_{n}=U_{n+1}$ for all $n$, thus completing the proof.
\end{proof}

Now we focus on $V_{0,\beta}^{\prime}\otimes L(c,0)$. If $c\neq c_{p,q}$,
$L(c,0)=V(c,0)/U(\operatorname*{Vir})L_{-1}v$ is a free module over algebra
$U(\operatorname*{Vir}_{-}\setminus\left\{  L_{-1}\right\}  )$ with standard
PBW basis
\[
\left\{  L_{-i_{n}}\cdots L_{-i_{1}}v:i_{n}\geq\cdots\geq i_{1}>1\right\}  ,
\]
where $v$ is a highest weight vector.

Let $\alpha=0$ and $\beta\neq0,1$. Then the relations
\begin{align*}
L_{-1}(v_{n}\otimes v)  &  =-n(v_{n-1}\otimes v),\\
L_{1}(v_{n}\otimes v)  &  =-(n+2\beta)(v_{n+1}\otimes v),
\end{align*}
or
\[
L_{-1}L_{2}(v_{n}\otimes v)=(n+2)\beta(v_{n+2}\otimes v)\neq0
\]
in case $n+2\beta=0$ show that
\[
\cdots=U_{-2}=U_{-1}\supseteq U_{0}=U_{1}=\cdots
\]
If $\beta=0$ there is no $v_{0}$ in $V_{0,0}^{\prime}$, so we have
\[
\cdots=U_{-2}=U_{-1}\supseteq U_{1}=U_{2}=\cdots
\]
and if $\beta=1$ there is no $v_{-1}$, hence
\[
\cdots=U_{-3}=U_{-2}\supseteq U_{0}=U_{1}=\cdots
\]

\begin{theorem}
\label{n}Let $c\neq c_{p,q}$ and $\alpha\in\mathbb{Z}$. Module $V=V_{\alpha
,\beta}^{\prime}\otimes L(c,0)$ is reducible and contains a nontrivial
submodule $U$ such that $V/U\cong V(c,h)$ where $h=1-\beta$ if $\beta\neq1$,
and $h=1$ if $\beta=1$.
\end{theorem}

\begin{proof}
Since $\alpha\in\mathbb{Z}$ we can assume $\alpha=0$. Suppose $V$ is
irreducible. Then $U_{-2}=U_{1}$ so there exists $x\in U(\operatorname*{Vir})$
such that $x(v_{1}\otimes v)=v_{-2}\otimes v$. Recall that
\[
x=\sum_{\substack{k_{1},\ldots,k_{n}\in(\mathbb{Z}_{+})^{n} \\n\in\mathbb{N}%
}}L_{-n}^{k_{n}}\cdots L_{-1}^{k_{1}}x_{k_{1}\cdots k_{n}}
\]
for some homogeneous $x_{k_{1}\cdots k_{n}}\in U(\operatorname*{Vir}%
\nolimits_{+})$. Since $L_{k}(v_{n}\otimes v)=L_{k}v_{n}\otimes v$ for $k>0$
and $L_{-1}(v_{0}\otimes v)=0$ we can write
\[
\sum_{k=0}^{m}x_{k+2}(v_{k}\otimes v)=v_{-2}\otimes v
\]
for some $x_{j}\in U(\operatorname*{Vir}_{-}\setminus\left\{  L_{-1}\right\}
)_{-j}$, $m\in\mathbb{N}$. But then $v_{m}\otimes x_{m+2}v$ must be zero,
leading to $x_{m+2}v=0$, which is a contradiction since $L(c,0)$ is free over
$U(\operatorname*{Vir}_{-}\setminus\left\{  L_{-1}\right\}  )$.

From Lemma \ref{kvoc} we know that $V/U_{1}$ is a highest weight module with
highest weight $1-\beta$ and cyclic generator $v_{-1}\otimes v$ (or highest
weight module with highest weight $1$ and generator $v_{-2}\otimes v$ if
$\beta=1$). Let us show this module is free over $U(\operatorname*{Vir}_{-})$.
For simplicity, suppose $\beta\neq0$. Notice that
\begin{align*}
U_{0}  &  =\left\{  u(v_{0}\otimes v):u\in U\left(  \operatorname*{Vir}%
\right)  \right\}  =\\
&  =\operatorname*{span}\nolimits_{%
\mathbb{C}
}\left\{  u^{\prime}(v_{k}\otimes v):u^{\prime}\in U\left(
\operatorname*{Vir}\nolimits_{-}\setminus\left\{  L_{-1}\right\}  \right)
,k\geq0\right\}
\end{align*}
and each $u^{\prime}(v_{k}\otimes v)$ contains a component $v_{k}\otimes
u^{\prime}v\neq0$. Suppose $V/U_{0}$ is not free. Then there exists $x\in
U\left(  \operatorname*{Vir}_{-}\right)  $ such that $x(v_{-1}\otimes v)\in
U_{0}$. But this is a contradiction since $x\left(  v_{-1}\otimes v\right)  $
can not have a component $v_{k}\otimes u^{\prime}v$ for $k\geq0$.
\end{proof}

\begin{remark}
For another proof of Theorem \ref{n} see Remark \ref{n2} on pg.\ \pageref{n2}.
\end{remark}

\begin{theorem}
\label{m}Let $c\neq c_{p,q}$ and $\alpha\in\mathbb{Z}$. Then
$U=U(\operatorname*{Vir})(v_{1}\otimes v)$ is an irreducible submodule in
$V_{\alpha,\beta}^{\prime}\otimes L(c,0)$, not isomorphic to some
$V_{\gamma,\delta}^{\prime}\otimes L(c,h)$.
\end{theorem}

\begin{proof}
Since $U=U_{k}$ for every $k\in\mathbb{N}$, irreducibility follows from
Corollary \ref{f1}.

Suppose there is a nontrivial $\operatorname*{Vir}$-homomorphism
$\Phi:U\rightarrow V_{\gamma,\delta}^{\prime}\otimes L(c,h)$ for
$\gamma,\delta,h\in\mathbb{C}$. Since $\operatorname*{Supp}U_{+}%
=-\beta+\mathbb{Z}$ and $\operatorname*{Supp}V_{\gamma,\delta}^{\prime}\otimes
L(c,h)=h-\gamma-\delta+\mathbb{Z}$, we have $h-\gamma-\delta+m=-\beta$ for
some $m\in\mathbb{Z}$. Let $V_{\gamma,\delta}^{\prime}=\bigoplus
\limits_{n\in\mathbb{Z}}\mathbb{C}w_{n}$ and $w$ a highest weight vector in
$L\left(  c,h\right)  $.

Let $\beta\neq0$. Then $U=U(\operatorname*{Vir})(v_{0}\otimes v)$ and
\[
\Phi\left(  v_{0}\otimes v\right)  =w_{0}\otimes x_{0}+w_{1}\otimes
x_{1}+\cdots+w_{n}\otimes x_{n}
\]
where $x_{j}\in L\left(  c,h\right)  _{j}$. Since $L_{-1}\left(  v_{0}\otimes
v\right)  =0$ we have $L_{-1}\Phi(v_{0}\otimes v)=0$, i.e.\
\begin{equation}
\sum_{i=0}^{n}\left(  i+\gamma\right)  w_{i-1}\otimes x_{i}=\sum_{i=0}%
^{n}w_{i}\otimes L_{-1}x_{i} \label{1}%
\end{equation}
Suppose $\gamma\in\mathbb{Z}$ i.e.\ $\gamma\neq0$. Then (\ref{1}) becomes
\[
\sum_{i=0}^{n-1}(i+1)w_{i}\otimes x_{i+1}=\sum_{i=0}^{n}w_{i}\otimes
L_{-1}x_{i}
\]
which leads to $x_{i+1}=\frac{1}{i+1}L_{-1}x_{i}$ for $i=0,1,\ldots,n-1$ and
$L_{-1}x_{n}=0$. Hence,
\[
x_{j}=\frac{1}{j!}L_{-1}^{j}x_{0}\text{ za }j=1,\ldots,n\text{.}
\]
We can assume $x_{0}=w$ which leads to $L_{-1}^{n}w=0$. This means that $h=0$
but then $V_{0,\delta}^{\prime}\otimes L(c,0)$ is reducible so $\Phi$ is not
an isomorphism. Therefore $\gamma\notin\mathbb{Z}$. Then from (\ref{1}) it
follows that $x_{0}=0$. Let $k\in\mathbb{N}$ be the smallest such that
$x_{k}\neq0$. Then (\ref{1}) becomes
\[
\sum_{i=k-1}^{n-1}w_{i}\otimes(i+1+\gamma)x_{i+1}=\sum_{i=k}^{n}w_{i}\otimes
L_{-1}x_{i}
\]
leading to $x_{k}=0$ a contradiction.

If $\beta=0$, the proof is essentially the same except that we consider
$v_{1}\otimes v$ instead of $v_{0}\otimes v$.
\end{proof}

Now we generalize Theorem \ref{n}.

\begin{theorem}
\label{poly}Let $c\neq c_{p,q}$ and suppose $V(c,h)$ is reducible with degree
$m$. Then a degree $m$ polynomial $p(x)\in\mathbb{Q}(\alpha,\beta,h)\left[
x\right]  $ exists such that $V_{\alpha,\beta}^{\prime}\otimes L(c,h)$ is
reducible if and only if $p$ has an integral root. For every integral root
$n$, there is a subquotient in $V_{\alpha,\beta}^{\prime}\otimes L(c,h)$ which
is isomorphic to a highest weight module of the highest weight either
$h-\alpha-\beta-n+1$ or $h+1$ in case $n=-\alpha\in\mathbb{Z}$ and $\beta=1$.
\end{theorem}

\begin{proof}
When $h=0$, we have $p(x)=-(x+\alpha)$, and theorem is essentialy a
combination of Proposition \ref{j} and Theorem \ref{n}. Therefore we assume
$m>1$.

First we find such polynomial $p$ using a singular vector $uv\in V(c,h)_{h+m}
$, $u\in U(\operatorname*{Vir}_{-})_{-m}$. It is well known that $u=L_{-1}%
^{m}+\sum q_{i_{1},\ldots,i_{n}}L_{-i_{n}}\cdots L_{-i_{1}}$, $i_{n}\geq
\cdots\geq i_{1}$, $i_{1}+\cdots+i_{n}=m$ for some $q_{i_{1},\ldots,i_{n}}%
\in\mathbb{Q}(c,h)$ and $q_{1,\ldots,1}=1$ (see \cite{Feigin-Fuchs}). Then
\[
u(v_{n+m-1}\otimes v)=p^{\prime}(n)v_{n-1}\otimes v+\sum a_{i}v_{n+m-i}\otimes
L_{-j_{k}}\cdots L_{-j_{1}}v
\]
where $p^{\prime}(n)\in\mathbb{Q}(\alpha,\beta,h)\left[  n\right]  $ is a
degree $m$ polynomial, $2\leq i\leq m$, and $0<k<m$ (since $uv=0$). We want to
eliminate the sum using induction on $k$. Let $\lambda=(n+m-i+\alpha
+(1-j)\beta)$. Then
\[
v_{n+m-i}\otimes L_{-j}v=L_{-j}(v_{n+m-i}\otimes v)+\lambda v_{n+m-i-j}\otimes
v
\]
and
\begin{gather*}
v_{n+m-i}\otimes L_{-j_{k}}\cdots L_{-j_{1}}v=\\
=L_{-j_{k}}(v_{n+m-i}\otimes L_{-j_{k-1}}\cdots L_{-j_{1}}v)+\lambda
v_{n+m-i-j}\otimes L_{-j_{k-1}}\cdots L_{-j_{1}}v.
\end{gather*}
Therefore, we get $u_{i}\in U(\operatorname*{Vir}_{-})_{-(m-i+1)}$ such that
\begin{equation}
u(v_{n+m-1}\otimes v)+\sum_{i=1}^{m-1}u_{i}(v_{n+m-1-i}\otimes v)=p(n)v_{n-1}%
\otimes v \label{sing}%
\end{equation}
with $p(n)\in\mathbb{Q}(\alpha,\beta,h,c)\left[  n\right]  $ a degree $m$
polynomial. This shows that
\begin{equation}
U_{n-1}\subseteq U_{n}+\cdots+U_{n+m-1} \label{rel}%
\end{equation}
if $p(n)\neq0$.

Next we prove irreducibility in case $p$ has no integral roots. Assume
$\alpha+2\beta\notin\mathbb{Z}$. Then
\[
L_{1}(v_{n}\otimes v)=-(n+\alpha+2\beta)v_{n+1}\otimes v\neq0
\]
so $U_{n}\supseteq U_{n+1}$ for all $n\in\mathbb{Z}$. Therefore (\ref{rel})
becomes $U_{n-1}\subseteq U_{n}$ and, if $p$ has no integral roots then
$U_{n}=U_{n+1}$ for all $n$. Therefore $V_{\alpha,\beta}^{\prime}\otimes
L(c,h)$ is irreducible. In case $\alpha=0$ and $\beta=0$ (resp.$\ \beta=1$),
$U_{0}$ (resp.$\ U_{-1}$) does not exist but $L_{2}(v_{n}\otimes v)$ shows
that $U_{-1}\supseteq U_{1}$ (resp.$\ U_{-2}\supseteq U_{0}$). Combined with
(\ref{rel}) this proves irreducibility.

Now suppose $\alpha+2\beta\in\mathbb{Z}$. Since $\alpha$ is invariant modulo
$\mathbb{Z}$ we may assume $\alpha+2\beta=0$. Now we have
\begin{equation}
\cdots\supseteq U_{-1}\supseteq U_{0},U_{1}\supseteq U_{2}\supseteq
\cdots\label{2}%
\end{equation}
If $\alpha=\beta=0$ (resp.$\ \alpha=-2$ and $\beta=1$), then $U_{0}$
(resp.$\ U_{1}$) does not exist so we do not have a problem. Let $\beta
\neq0,1$ and suppose $p(2)\neq0$. Then (\ref{rel}) shows that $U_{1}\subseteq
U_{2}$ which combined with (\ref{2}) implies $U_{1}=U_{2}\subseteq U_{0}$.
Specially, if $p$ has no integral roots we get $U_{n}=U_{n+1}$ for all $n$, so
$V_{\alpha,\beta}^{\prime}\otimes L(c,h)$ is irreducible.

Now let $p(n)=0$ for some $n\in\mathbb{Z}$. Suppose $v_{n-1}\otimes
v=x(v_{n}\otimes v)$ for $x\in U(\operatorname*{Vir}_{-})$. Note that we can
write
\[
v_{n-1}\otimes v=\sum_{i=0}^{k-1}x_{i+1}(v_{n+i}\otimes v)
\]
for some $k\in\mathbb{N}$, and $x_{j}\in U(\operatorname*{Vir}_{-})_{-j}$. But
then $v_{n+k}\otimes x_{k}v=0$ hence $x_{k}\in U(\operatorname*{Vir}_{-})u$
which implies $k\geq m$ and $x_{k}=y_{k-m}u$ for some $y_{k-m}\in
U(\operatorname*{Vir}_{-})_{m-k}$. Suppose $k>m$. Then we can apply
(\ref{sing}) to show that
\[
x_{k}(v_{n+k-1}\otimes v)=y_{k-m}u(v_{n+k-1}\otimes v)=\sum_{i=1}^{m-1}%
y_{k-m}u_{i}(v_{n+k-1-i}\otimes v)
\]
thus we can write $v_{n-1}\otimes v=\sum_{i=0}^{k-2}x_{i+1}(v_{n+1}\otimes v)
$. Proceeding by induction, we conclude that $k=m$ and this leads back to
(\ref{sing}). However, $p(n)=0$ so we conclude that the equation
$v_{n-1}\otimes v=x(v_{n}\otimes v)$ has no solution in $U(\operatorname*{Vir}%
_{-})$. Therefore $U_{n-1}\subsetneq V_{\alpha,\beta}^{\prime}\otimes L(c,h)$
and $U_{n-1}/U_{n}$ is the highest weight module with highest weight
$h-\alpha-\beta-n+1$ by Lemma \ref{kvoc}. In case $\alpha\in\mathbb{Z}$,
$\beta=1$ and $n=-\alpha$, there is no $U_{n-1}$ and $U_{n-2}/U_{n}$ is the
highest weight module generated by highest weight $h+1$ vector $v_{-2}\otimes
v$. This completes the proof.
\end{proof}

\begin{remark}
Claims of Theorem \ref{poly} have been proved in Theorem 1 a) in
\cite{Chen-Guo-Zhao}. Polynomial $p$ is closely related to $\varphi_{n}$ in
\cite{Chen-Guo-Zhao}. Respective to terminology and definitions of the
Virasoro algebra and intermediate series used in this paper, $p$ can be
regarded as a linear map $\varphi_{n}:U\left(  \operatorname*{Vir}_{-}\right)
\rightarrow\mathbb{C}$ defined by
\[
\varphi_{n}\left(  L_{-k_{r}}\cdots L_{-k_{1}}\right)  =%
{\textstyle\prod\limits_{j=1}^{r}}
\left(  \alpha+\left(  1-k_{j}\right)  \beta+n+\sum_{i=1}^{j}k_{i}-1\right)
\]

\end{remark}

\begin{example}
\label{s2}If $c=\frac{10h-16h^{2}}{1+2h}$, a singular vector in $V(c,h)$ is
$s_{2}v=(L_{-1}^{2}-\frac{4h+2}{3}L_{-2})v$. By direct computation we get
\begin{gather}
\left(  \frac{-1}{n+\alpha+2\beta}s_{2}L_{1}+2\left(  n+1+\alpha\right)
L_{-1}\right)  \left(  v_{n}\otimes v\right)  =\label{sg}\\
=-\left(  \left(  n+1+\alpha\right)  \left(  n+\alpha\right)  -\frac{4h+2}%
{3}\left(  n+1+\alpha-\beta\right)  \right)  v_{n-1}\otimes v.\nonumber
\end{gather}
The roots of $p\left(  n\right)  =\left(  n+1+\alpha\right)  \left(
n+\alpha\right)  -\frac{4h+2}{3}\left(  n+1+\alpha-\beta\right)  $ are
$-\alpha+\frac{4h-1}{6}\pm\frac{1}{6}\sqrt{\left(  4h+5\right)  ^{2}%
-24\beta\left(  2h+1\right)  }$.
\end{example}

\begin{example}
\label{s3}For $m=3$, the roots of polynomial $p$ are $-\alpha+h$ and
$-\alpha+\frac{h-1}{2}\pm\frac{1}{2}\sqrt{\left(  h+3\right)  ^{2}%
-8\beta\left(  h+1\right)  }$.
\end{example}

\begin{corollary}
\label{sing2}Suppose Verma module $V\left(  c,h\right)  $ contains a weight
$2$ singular vector and $-\alpha+\frac{4h-1\pm\sqrt{\left(  4h+5\right)
^{2}-24\beta\left(  2h+1\right)  }}{6}\notin\mathbb{Z}$. Then $V_{\alpha
,\beta}^{\prime}\otimes L\left(  c,h\right)  $ is irreducible.
\end{corollary}

\begin{corollary}
\label{sing3}If Verma module $V(c,h)$ has a weight $3$ singular vector and
$-\alpha+\frac{h-1\pm\sqrt{\left(  h+3\right)  ^{2}-8\beta\left(  h+1\right)
}}{2},-\alpha+h\notin\mathbb{Z}$, module $V_{\alpha,\beta}^{\prime}\otimes
L(c,h) $ is irreducible.
\end{corollary}

\begin{remark}
When $c=c_{p,q}$, reducible Verma module $V(c,h)$ has two independent singular
vectors. Each of them produces two polynomials as in Theorem \ref{poly}.
$V_{\alpha,\beta}^{\prime}\otimes L(c,h)$ is irreducible unless there is a
common integral root. Proving reducibility directly, like in Theorem
\ref{poly} seems somewhat challenging. However, this will follow from the
existence of intertwining operators, as we will see in the following section.
\end{remark}

\section{Intertwining operators and reducibility of $V_{\alpha,\beta}^{\prime
}\otimes L(c,h)$\label{sec}}

Vertex operator algebras (VOAs) are a fundamental class of algebraic
structures which have arisen in mathematics and physics a few decades ago.
Their importance is supported by their numerous relations with many fields of
algebra, representation theory, topology, differential equations and conformal
field theory. The main original motivation for the introduction of the notion
of VOA arose from the problem of realizing the monster sporadic group as a
symmetry group of a certain infinite-dimensional vector space (\cite{Frenkel -
Lepowsky - Meurman}).

An interested reader should consult \cite{Frenkel - Huang - Lepowsky} or
\cite{Lepowsky-Li} for a detailed approach. Here we present only basic
definitions which should suffice our needs.

For any algebraic expression $z$ we set $\delta\left(  z\right)  =\sum
_{n\in\mathbb{Z}}z^{n}$ provided that this sum makes sense. This is the formal
analogue of the $\delta$-distribution at $z=1$; in particular, $\delta\left(
z\right)  f\left(  z\right)  =\delta\left(  z\right)  f\left(  1\right)  $ for
any $f$ for which these expressions are defined.

\begin{definition}
A \textbf{vertex operator algebra} $\left(  V,Y,\mathbf{1}\right)  $ is a
$\mathbb{Z}$-graded vector space (graded by \textit{weights}) $V=%
{\textstyle\bigoplus\limits_{n\in\mathbb{Z}}}
V_{\left(  n\right)  }$ such that $\dim V_{\left(  n\right)  }<\infty$ for
$n\in\mathbb{Z}$ and $V_{\left(  n\right)  }=0$ for $n$ sufficiently small,
equipped with a linear map $V\otimes V\rightarrow V\left[  \left[
z,z^{-1}\right]  \right]  $, or equivalently,
\begin{align*}
V  &  \rightarrow\left(  \operatorname*{End}V\right)  \left[  \left[
z,z^{-1}\right]  \right] \\
v  &  \mapsto Y\left(  v,z\right)  =\sum_{n\in\mathbb{Z}}v_{n}z^{-n-1}\text{
(where }v_{n}\in\operatorname*{End}V\text{),}%
\end{align*}
$Y\left(  v,z\right)  $ denoting the \textit{vertex operator associated} with
$v$, and equipped with two distinguished homogeneous vectors $\mathbf{1}$ (the
\textit{vacuum}) and $\omega\in V$. The following conditions are assumed for
$u,v\in V:$%
\[
u_{n}v=0\text{ for }n\text{ sufficiently large;}
\]%
\[
Y\left(  \mathbf{1},z\right)  =1\left(  =\operatorname*{id}\nolimits_{V}%
\right)  ;
\]
the \textit{creation property} holds:
\[
Y\left(  v,z\right)  \mathbf{1}\in V\left[  \left[  z\right]  \right]  \text{
and }\lim_{z\rightarrow0}Y\left(  v,z\right)  \mathbf{1}=0
\]
(that is, $Y\left(  v,z\right)  \mathbf{1}$ involves only nonnegative integral
powers of $z$ and the constant term is $v$); the \textit{Jacobi identity}:
\begin{align}
&  z_{0}^{-1}\delta\left(  \frac{z_{1}-z_{2}}{z_{0}}\right)  Y(u,z_{1}%
)Y(v,z_{2})-z_{0}^{-1}\delta\left(  \frac{z_{2}-z_{1}}{-z_{0}}\right)
Y(v,z_{2})Y(u,z_{1})\label{jacobi}\\
&  =z_{2}^{-1}\delta\left(  \frac{z_{1}-z_{0}}{z_{2}}\right)  Y(Y(u,z_{0}%
)v,z_{2});\nonumber
\end{align}
the Virasoro algebra relations:
\[
\left[  L_{m},L_{n}\right]  =\left(  m-n\right)  L_{m+n}+\frac{m^{3}-m}%
{12}\delta_{m+n,0}\left(  \text{rank }V\right)
\]
for $m,n\in\mathbb{Z}$, where
\[
L_{n}=\omega_{n+1}\text{ for }n\in\mathbb{Z}\text{, i.e.\ }Y\left(
\omega,z\right)  =\sum_{n\in\mathbb{Z}}L_{n}z^{-n-2}
\]
and rank~$V\in\mathbb{C}$, $L_{0}v=nv=\left(  \operatorname{wt}v\right)  v$
for $n\in\mathbb{Z}$ and $v\in V_{\left(  n\right)  };$
\[
\frac{d}{dz}Y\left(  v,z\right)  =Y\left(  L_{-1}v,z\right)
\]
(the $L_{-1}$-derivative property).
\end{definition}

\begin{definition}
Given a VOA $\left(  V,Y,\mathbf{1}\right)  ,$ a \textbf{module} $\left(
W,\mathcal{Y}\right)  $ for $V$ is a $\mathbb{Q}$-graded vector space $W=%
{\textstyle\bigoplus\limits_{n\in\mathbb{Q}}}
W_{\left(  n\right)  }$ such that $\dim W_{\left(  n\right)  }<\infty$ for
$n\in\mathbb{Q}$, and $W_{\left(  n\right)  }=0$ for $n$ sufficiently small,
equipped with a linear map $V\otimes W\rightarrow W\left[  \left[
z,z^{-1}\right]  \right]  $, or equivalently,
\begin{align*}
V  &  \rightarrow\left(  \operatorname*{End}W\right)  \left[  \left[
z,z^{-1}\right]  \right] \\
v  &  \mapsto\mathcal{Y}\left(  v,z\right)  =\sum_{n\in\mathbb{Z}}%
v_{n}z^{-n-1}\text{ (where }v_{n}\in\operatorname*{End}W\text{),}%
\end{align*}
$\mathcal{Y}\left(  v,z\right)  $ denoting the \textit{vertex operator
associated} with $v$. The Virasoro algebra relations hold on $W$ with scalar
equal to rank $V$:
\[
\left[  L_{m},L_{n}\right]  =\left(  m-n\right)  L_{m+n}+\frac{m^{3}-m}%
{12}\delta_{m+n,0}\left(  \text{rank }V\right)
\]
for $m,n\in\mathbb{Z}$, where
\[
L_{n}=\omega_{n+1}\text{ for }n\in\mathbb{Z}\text{, i.e.\ }\mathcal{Y}\left(
\omega,z\right)  =\sum_{n\in\mathbb{Z}}L_{n}z^{-n-2}
\]
$L_{0}w=nw$ for $n\in\mathbb{Q}$ and $w\in W_{\left(  n\right)  }.$ For
$u,v\in V$ and $w\in W$ the following properties hold:

\begin{description}
\item[(i)] \textit{Truncation property}: $v_{n}w=0$ for $n$ sufficiently large
and $Y\left(  \mathbf{1},z\right)  =1$;

\item[(ii)] $L_{-1}$-derivative property
\[
\frac{d}{dz}\mathcal{Y}\left(  v,z\right)  =\mathcal{Y}\left(  L_{-1}%
v,z\right)
\]

\item[(iii)] The Jacobi identity
\begin{align*}
&  z_{0}^{-1}\delta\left(  \frac{z_{1}-z_{2}}{z_{0}}\right)  \mathcal{Y}%
(u,z_{1})\mathcal{Y}(v,z_{2})-z_{0}^{-1}\delta\left(  \frac{z_{2}-z_{1}%
}{-z_{0}}\right)  \mathcal{Y}(v,z_{2})\mathcal{Y}(u,z_{1})\\
&  =z_{2}^{-1}\delta\left(  \frac{z_{1}-z_{0}}{z_{2}}\right)  \mathcal{Y}%
(Y(u,z_{0})v,z_{2});
\end{align*}
($Y\left(  u,z_{0}\right)  $ is the operator associated with $V$).
\end{description}
\end{definition}

\begin{definition}
Let $V=(V,Y,\mathbf{1})$ be a vertex operator algebra, and $\left(
W_{i},Y_{i}\right)  $, $i=1,2,3$ three (not necessarily distinct) $V$-modules.
\textbf{Intertwining operator} of type $\binom{W_{3}}{W_{1}\text{\quad}W_{2}}$
is a linear map $W_{1}\otimes W_{2}\rightarrow W_{3}\{z\}=\left\{  \sum_{n\in%
\mathbb{Q}
}u_{n}z^{n}:u_{n}\in W_{k}\right\}  $, or equivalently,
\begin{align*}
W_{1}  &  \rightarrow\left(  \operatorname*{Hom}\left(  W_{2},W_{3}\right)
\right)  \left\{  z\right\} \\
w  &  \mapsto\mathcal{I}\left(  w,z\right)  =\sum_{n\in\mathbb{Q}}%
w_{n}z^{-n-1}\text{, with }w_{n}\in\operatorname*{Hom}\left(  W_{2}%
,W_{3}\right)  .
\end{align*}
For any $v\in V$, $u\in W_{1}$, $w\in W_{2}$ the following conditions are satisfied:

\begin{description}
\item[(i)] \textit{Truncation property} - $u_{n}v=0$ for $n$ sufficiently large;

\item[(ii)] $L_{-1}$\textit{-derivative property} - $\mathcal{I}%
(L_{-1}u,z)=\frac{d}{dz}\mathcal{I}(u,z)$;

\item[(iii)] The \textit{Jacobi identity}
\begin{align*}
&  z_{0}^{-1}\delta\left(  \frac{z_{1}-z_{2}}{z_{0}}\right)  Y_{3}%
(v,z_{1})\mathcal{I}(u,z_{2})w-z_{0}^{-1}\delta\left(  \frac{z_{2}-z_{1}%
}{-z_{0}}\right)  \mathcal{I}(u,z_{2})Y_{2}(v,z_{1})w\\
&  =z_{2}^{-1}\delta\left(  \frac{z_{1}-z_{0}}{z_{2}}\right)  \mathcal{I}%
(Y_{1}(v,z_{0})u,z_{2})w.
\end{align*}

\end{description}
\end{definition}

\bigskip

\bigskip

Now, if $M_{i}$, $i=1,2,3$ are highest weight $\operatorname*{Vir}$-modules of
highest weights $h_{i}$ and central charge $c$, and $\mathcal{I}$ is an
intertwining operator of type $\binom{M_{3}}{M_{1}\text{\quad}M_{2}}$, then
$\mathcal{I}(u,z)=z^{-\alpha}\sum_{n\in\mathbb{Z}}u_{n}z^{-n-1}$ where
$\alpha=h_{1}+h_{2}-h_{3}$ \cite{Frenkel-Zhu}. Suppose there exists such a
nontrivial intertwining operator $\mathcal{I}\ $and suppose $h_{1}\neq0$. Let
$v$ the highest weight vector in $M_{1}$. Equating the coefficient of
$z_{0}^{-1}z_{1}^{-m-1}z_{2}^{-n-1}$ in (\ref{jacobi}) yields
\[
\left[  u_{m},v_{n}\right]  =\sum_{i\geq0}\binom{m}{i}\left(  u_{i}v\right)
_{m+n-i},
\]
and in particular for $u=\omega$ we have
\begin{align*}
\left[  L_{m},v_{n}\right]   &  =\sum_{i\geq0}\binom{m+1}{i}(L_{i-1}%
v)_{m+n-i+1}=\\
&  =(L_{-1}v)_{m+n+1}+(m+1)(L_{0}v)_{m+n}=\\
&  =-(\alpha+n+m+1)v_{m+n}+(m+1)h_{1}v_{m+n}=\\
&  =-(n+\alpha+(1-h_{1})(1+m))v_{m+n},
\end{align*}
hence the components of $\mathcal{I}(v,z)$ span $V_{\alpha,\beta}^{\prime}$,
where $\beta=1-h_{1}$. Moreover, a nontrivial $\operatorname*{Vir}%
$-homomorphism $\Phi$ is defined:
\[
\Phi:V_{\alpha,\beta}^{\prime}\otimes M_{2}\rightarrow M_{3},\qquad\Phi
(v_{n}\otimes v^{\prime})=v_{n}v^{\prime}%
\]
where $v^{\prime}$ is the highest weight vector in $M_{2}$. Comparing the
dimensions of the weight subspaces we conclude that $V_{\alpha,\beta}^{\prime
}\otimes M_{2}$ is reducible.

\begin{theorem}
\label{red}Let $M_{i}$, $i=1,2,3$ be the highest weight $\operatorname*{Vir}%
$-modules with the highest weight $(c,h_{i})$, and $h_{1}\neq0$. Suppose a
nontrivial intertwining operator $I$ of type $\binom{M_{3}}{M_{1}\text{\quad
}M_{2}}$ exists. Then there exists nontrivial $\operatorname*{Vir}%
$-homomorphism $V_{\alpha,\beta}^{\prime}\otimes M_{2}\rightarrow M_{3}$,
where $\alpha=h_{1}+h_{2}-h_{3}$ and $\beta=1-h_{1}$. Consequently,
$V_{\alpha,\beta}^{\prime}\otimes M_{2}$ is reducible.
\end{theorem}

\subsection{Minimal models\label{min}}

$\operatorname*{Vir}$-module $L(c,0)$ admits an irreducible VOA structure. For
$c=c_{p,q}=1-6\frac{(p-q)^{2}}{pq},$ where $p,q>1$ are relatively prime, this
algebra is rational (\cite{Frenkel-Zhu}) i.e.\ it has only finitely many
irreducible modules and every finitely generated module is a direct sum of
irreducibles. Let
\begin{equation}
h_{m,n}=\frac{\left(  np-mq\right)  ^{2}-\left(  p-q\right)  ^{2}}%
{4pq},0<m<p,0<n<q. \label{h min}%
\end{equation}
Then $L(c_{p,q},h_{m,n})$ is a module over VOA $L(c_{p,q},0)$ called a
\textbf{minimal model}. Wang has shown in \cite{Wang} that minimal models are
all irreducibles for VOA $L(c_{p,q},0)$.

Using fusion rules (\cite{Frenkel-Zhu}, \cite{Wang}) we have the complete list
of intertwining operators for VOA $L(c_{p,q},0)$.

An ordered triple of pairs of integers $((m,n),\allowbreak(m^{\prime
},n^{\prime}),\allowbreak(m^{\prime\prime},n^{\prime\prime}))$ is
\textbf{admissible} if $0<m_{1},m_{2},m_{2}<p$, $0<n_{1},n_{2},n_{3}<q$,
$m_{1}+m_{2}+m_{3}<2p$, $n_{1}+n_{2}+n_{3}<2q$, $m_{1}<m_{2}+m_{3}$,
$m_{2}<m_{1}+m_{3}$, $m_{3}<m_{1}+m_{2}$, $n_{1}<n_{2}+n_{3}$, $n_{2}%
<n_{1}+n_{3}$, $n_{3}<n_{1}+n_{2}$ and the sums $m_{1}+m_{2}+m_{3}$ and
$n_{1}+n_{2}+n_{3}$ are odd. We identify the triples $((m_{1},n_{1}%
),\allowbreak(m_{2},n_{2}),\allowbreak(m_{3},n_{3}))$ and $((m_{1}%
,n_{1}),\allowbreak(p-m_{2},q-n_{2}),\allowbreak(p-m_{3},q-n_{3}))$.

Let $c=c_{p,q}$. A nontrivial intertwining operator of type $\binom
{L(c,h_{m^{\prime\prime},n^{\prime\prime}})}{L(c,h_{m,n})\text{\quad
}L(c,h_{m^{\prime},n^{\prime}})}$ exists if and only if $((m,n),(m^{\prime
},n^{\prime}),(m^{\prime\prime},n^{\prime\prime}))$ is an admissible triple of
pairs. (See \cite{Wang}).

If $h=h_{m,n}$ as in (\ref{h min}), the following intertwining operators
always exist:
\begin{align*}
&  \binom{L(c,h)}{L(c,0)\text{\quad}L(c,h)}\text{ (module }L(c,h)\text{)}\\
&  \binom{L(c,h)}{L(c,h)\text{\quad}L(c,0)}\text{ (transposed operator)}\\
&  \binom{L(c,0)}{L(c,h)\text{\quad}L(c,h)}\text{ (adjoint operator)}%
\end{align*}

\begin{example}
[Yang-Lee model]$L(-\frac{22}{5},0)$ and $L(-\frac{22}{5},-\frac{1}{5})$ are
the only irreducibles over VOA$\ L(-\frac{22}{5},0)$. Operator of type
$\binom{L(-\frac{22}{5},-\frac{1}{5})}{L(-\frac{22}{5},-\frac{1}%
{5})\text{\quad}L(-\frac{22}{5},-\frac{1}{5})}$ completes the list of
intertwining operators for $L(-\frac{22}{5},0)$.
\end{example}

\begin{example}
[Ising model]$L(\frac{1}{2},0)$, $L(\frac{1}{2},\frac{1}{16})$ and $L(\frac
{1}{2},\frac{1}{2})$ are the only irreducibles over $L(\frac{1}{2},0)$.
Operators of type $\binom{L(\frac{1}{2},\frac{1}{16})}{L(\frac{1}{2},\frac
{1}{2})\text{\quad}L(\frac{1}{2},\frac{1}{16})}$, $\binom{L(\frac{1}{2}%
,\frac{1}{16})}{L(\frac{1}{2},\frac{1}{16})\text{\quad}L(\frac{1}{2},\frac
{1}{2})}$ and $\binom{L(\frac{1}{2},\frac{1}{2})}{L(\frac{1}{2},\frac{1}%
{16})\text{\quad}L(\frac{1}{2},\frac{1}{16})}$ complete the list of
intertwining operators for $c=\frac{1}{2}$.
\end{example}

\begin{corollary}
\label{hom}There are nontrivial $\operatorname*{Vir}$-homomorphisms:\newline%
$\left(  V_{0,\frac{6}{5}}^{\prime}\otimes L(-\frac{22}{5},0)\right)
\rightarrow L(-\frac{22}{5},-\frac{1}{5})$,\newline$\left(  V_{-\frac{2}%
{5},\frac{6}{5}}^{\prime}\otimes L(-\frac{22}{5},-\frac{1}{5})\right)
\rightarrow L(-\frac{22}{5},0)$, $\left(  V_{-\frac{1}{5},\frac{6}{5}}%
^{\prime}\otimes L(-\frac{22}{5},-\frac{1}{5})\right)  \rightarrow
L(-\frac{22}{5},-\frac{1}{5})$;\newline$\left(  V_{0,\frac{1}{2}}^{\prime
}\otimes L(\frac{1}{2},0)\right)  \rightarrow L(\frac{1}{2},\frac{1}{2})$,
$\left(  V_{0,\frac{15}{16}}^{\prime}\otimes L(\frac{1}{2},0)\right)
\rightarrow L(\frac{1}{2},\frac{1}{16})$,\newline$\left(  V_{0,\frac{1}{2}%
}^{\prime}\otimes L(\frac{1}{2},\frac{1}{2})\right)  \rightarrow L(\frac{1}%
{2},0)$, $\left(  V_{\frac{1}{2}.\frac{15}{16}}^{\prime}\otimes L(\frac{1}%
{2},\frac{1}{2})\right)  \rightarrow L(\frac{1}{2},\frac{1}{16})$%
,\newline$\left(  V_{\frac{1}{8},\frac{15}{16}}^{\prime}\otimes L(\frac{1}%
{2},\frac{1}{16})\right)  \rightarrow L(\frac{1}{2},0)$, $\left(  V_{\frac
{1}{2},\frac{1}{2}}^{\prime}\otimes L(\frac{1}{2},\frac{1}{16})\right)
\rightarrow L(\frac{1}{2},\frac{1}{16})$,\newline$\left(  V_{-\frac{3}%
{8},\frac{15}{16}}^{\prime}\otimes L(\frac{1}{2},\frac{1}{16})\right)
\rightarrow L(\frac{1}{2},\frac{1}{2})$.
\end{corollary}

\begin{proof}
Directly from Theorem \ref{red}.
\end{proof}

\begin{remark}
\label{n2}Let $c\neq c_{p,q}$ and $h\neq0$. Then $V(c,h)$ is a module over VOA
$L(c,0)$ so there exists a transposed intertwining operator of type
$\binom{V(c,h)}{V(c,h)\text{\quad}L(c,0)}$. Therefore, a nontrivial
$\operatorname*{Vir}$-epimorphism $V_{\alpha,\beta}^{\prime}\otimes
L(c,0)\rightarrow V(c,h)$ exists which proves Theorem \ref{n}. However,
$V(c_{p,q},h)$ is not $L(c_{p,q},0)$-module.
\end{remark}

\bigskip Next we consider irreducibility of $V_{\alpha,\beta}^{\prime}\otimes
L(c_{p,q},h_{m,n})$, where $c=-\frac{22}{5},\frac{1}{2}$. We take advantage of
the fact that $V(c_{p,q},h_{m,n})$ has two singular vectors that we can use
independently. This way we prove irreducibility for all pairs $(\alpha,\beta)$
except those laying on intersection of two algebraic curves. These exceptions
are precisely those listed in Corollary \ref{hom}.

\begin{proposition}
\label{p1}Module $V_{\alpha,\beta}^{\prime}\otimes L(-\frac{22}{5},0)$ is
irreducible if and only if $(\alpha,\beta)\neq(0,\frac{6}{5})$. Moreover,
\[
\left(  V_{0,\frac{6}{5}}^{\prime}\otimes L(-\frac{22}{5},0)\right)
/U_{0}\cong L(-\frac{22}{5},-\frac{1}{5}).
\]

\end{proposition}

\begin{proof}
If $\alpha\notin\mathbb{Z}$, this is a special case of Proposition \ref{j}.
Let $\alpha=0$ and let $s=L_{-2}^{2}-\frac{3}{5}L_{-4}$. Then $sv=0$ if $v$ is
the highest weight vector in $L(-\frac{22}{5},0)$. In the discussion
preceeding Theorem \ref{n} we noted that
\[
\cdots=U_{-2}=U_{-1}\supseteq U_{0}=U_{1}=\cdots
\]
It remains to show that $U_{-1}\subseteq U_{0}$ (or $U_{-2}\subseteq U_{1}$ if
$\beta=0,1$). We have
\[
s(v_{3}\otimes v)+2(3-\beta)L_{-2}(v_{1}\otimes v)=\left(  \beta-\frac{6}%
{5}\right)  (1-\beta)v_{-1}\otimes v.
\]
Since $v_{3}\otimes v\in U_{3}=U_{1}$ we obtain $U_{1}\subseteq U_{-1}$ for
$\beta\neq1,\frac{6}{5}$. On the other hand, for $\beta\neq0$ we have
\[
s(v_{2}\otimes v)-2(\beta-2)L_{-2}(v_{0}\otimes v)=\left(  \frac{6}{5}%
-\beta\right)  (\beta+1)v_{-2}\otimes v,
\]
proving that $U_{0}\subseteq U_{-2}$, for $\beta\neq-1,0,\frac{6}{5}$. This
proves irreducibility when $\beta\neq\frac{6}{5}$.

Since $V_{0,\frac{6}{5}}^{\prime}\otimes L(-\frac{22}{5},0)$ is reducible by
Corollary \ref{hom}, then by Theorem \ref{main} we must have $U_{0}\neq
U_{-1}$. From lema \ref{kvoc} we know that $\left(  V_{0,\frac{6}{5}}^{\prime
}\otimes L(-\frac{22}{5},0)\right)  /U_{0}$ is isomorphic to some quotient of
$V(-\frac{22}{5},-\frac{1}{5})$. Next we check relations for singular vectors
in $V(-\frac{22}{5},-\frac{1}{5})$. We need to show that
\[
(L_{-1}^{2}-\frac{2}{5}L_{-2})(v_{-1}\otimes v)\in U_{0}
\]%
\[
(L_{-1}^{3}-\frac{8}{5}L_{-2}L_{-1}-\frac{4}{25}L_{-3})(v_{-1}\otimes v)\in
U_{0}
\]
By direct computation we see that
\[
(L_{-1}^{2}-\frac{2}{5}L_{-2})(v_{-1}\otimes v)=-s(v_{1}\otimes v)\in
U_{1}=U_{0}
\]
and
\[
(L_{-1}^{3}-\frac{8}{5}L_{-2}L_{-1}-\frac{4}{25}L_{-3})(v_{-1}\otimes v)=
\]%
\[
=-\frac{75}{2}(L_{-1}^{2}-\frac{2}{5}L_{-2})(v_{0}\otimes v)-25L_{-1}%
(L_{-1}^{2}-\frac{2}{5}L_{-2})(v_{1}\otimes v)\in U_{0}
\]
This shows irreducibility of $\left(  V_{0,\frac{6}{5}}^{\prime}\otimes
L(-\frac{22}{5},0)\right)  /U_{0}$ and completes the proof.
\end{proof}

\begin{proposition}
\label{p2}$V_{\alpha,\beta}^{\prime}\otimes L(\frac{1}{2},0)$ is irreducible
if and only if $(\alpha,\beta)\neq(0,\frac{1}{2}),(0,\frac{15}{16})$.
Moreover
\begin{align*}
\left(  V_{0,\frac{1}{2}}^{\prime}\otimes L(\frac{1}{2},0)\right)  /U_{0}  &
\cong L(\frac{1}{2},\frac{1}{2}),\\
\left(  V_{0,\frac{15}{16}}^{\prime}\otimes L(\frac{1}{2},0)\right)  /U_{0}
&  \cong L(\frac{1}{2},\frac{1}{16}).
\end{align*}

\end{proposition}

\begin{proof}
Again, we may assume $\alpha=0$, and need to check that $U_{-2}\subseteq
U_{1}$.

Let $s^{\prime}=64L_{-2}^{3}+93L_{-3}^{2}-264L_{-4}L_{-2}-108L_{-6}$. Then
$s^{\prime}v=0$ in $L\left(  \frac{1}{2},0\right)  $. The relations
\begin{gather*}
s^{\prime}(v_{5}\otimes v)+192(5-\beta)L_{-2}^{2}(v_{3}\otimes v)-264(5-\beta
)L_{-4}(v_{3}\otimes v)+\\
+186(5-2\beta)L_{-3}(v_{2}\otimes v)-264(5-3\beta)L_{-2}(v_{1}\otimes v)+\\
+192(5-\beta)(3-\beta)L_{-2}(v_{1}\otimes v)=-2(\beta-1)\left(  \beta-\frac
{1}{2}\right)  \left(  \beta-\frac{15}{16}\right)  (v_{-1}\otimes v)
\end{gather*}
and
\begin{gather*}
s^{\prime}(v_{4}\otimes v)+192(4-\beta)L_{-2}^{2}(v_{2}\otimes v)-264(4-\beta
)L_{-4}(v_{2}\otimes v)+\\
+186(4-2\beta)L_{-3}(v_{1}\otimes v)-264(4-3\beta)L_{-1}(v_{0}\otimes v)+\\
+192(4-\beta)(2-\beta)L_{-1}(v_{0}\otimes v)=2\left(  \beta+2\right)  \left(
\beta-\frac{1}{2}\right)  \left(  \beta-\frac{15}{16}\right)  \left(
v_{-2}\otimes v\right)
\end{gather*}
show that $U_{1}\subseteq U_{-2}$ when $\beta\neq\frac{1}{2},\frac{15}{16}$.

Let $\beta=\frac{1}{2}.$ Using $L_{1},$ $L_{2}$ and $L_{-1}$ we get
\[
\left(  V_{0,\frac{1}{2}}^{\prime}\otimes L(\frac{1}{2},0)\right)
=U_{-1}\supseteq U_{0}=U_{1}=\cdots
\]
Since $V_{0,\frac{1}{2}}^{\prime}\otimes L(\frac{1}{2},0)$ is reducible by
Corollary \ref{hom}, we must have $U_{0}\neq U_{1}$. By Lemma \ref{kvoc},
$U_{-1}/U_{0}$ is isomorphic to some quotient of $V(\frac{1}{2},\frac{1}{2})$.
It remains to show that
\begin{align*}
(L_{-1}^{2}-\frac{4}{3}L_{-2})(v_{-1}\otimes v)  &  \in U_{0},\\
(L_{-3}-\frac{4}{5}L_{-2}L_{-1})(v_{-1}\otimes v)  &  \in U_{0}.
\end{align*}
However, by direct computation we find
\begin{gather*}
(L_{-1}^{2}-\frac{4}{3}L_{-2})(v_{-1}\otimes v)=-\frac{4}{3}v_{-1}\otimes
L_{-2}v=s_{6}(v_{3}\otimes v)+\\
+480L_{-2}^{2}(v_{1}\otimes v)+372L_{-3}(v_{0}\otimes v)-660L_{-4}%
(v_{1}\otimes v)\in U_{0}%
\end{gather*}
where $s_{6}v$ is a weight 6 singular vector in $V(\frac{1}{2},0)$. Also,
$(L_{-3}-\frac{4}{5}L_{-2}L_{-1})(v_{-1}\otimes v)=0$.

If $\beta=\frac{15}{16}$, in a similar fashion one can show that
\begin{align*}
(L_{-1}^{2}-\frac{3}{4}L_{-2})(v_{-1}\otimes v)  &  \in U_{0},\\
(16L_{-2}^{2}-24L_{-3}L_{-1}-9L_{-4})(v_{-1}\otimes v)  &  \in U_{0}.
\end{align*}

\end{proof}

When $h_{m,n}\neq0$ we proceed similary. Using $L_{1},$ $L_{2}$ and singular
vectors on levels 2, 3 or 4 one can prove

\begin{proposition}
\label{p3}$V_{\alpha,\beta}^{\prime}\otimes L(-\frac{22}{5},-\frac{1}{5})$ is
irreducible if and only if $(\alpha,\beta)\neq(-\frac{2}{5},\frac{6}%
{5}),\allowbreak(-\frac{1}{5},\frac{6}{5})$.

$V_{\alpha,\beta}^{\prime}\otimes L(\frac{1}{2},\frac{1}{2})$ is irreducible
if and only if $(\alpha,\beta)\neq(0,\frac{1}{2}),(\frac{1}{2},\frac{15}{16})
$.

$V_{\alpha,\beta}^{\prime}\otimes L(\frac{1}{2},\frac{1}{16})$ is irreducible
if and only if $(\alpha,\beta)\neq(\frac{1}{8},\frac{15}{16}),\allowbreak
(-\frac{3}{8},\frac{15}{16}),\allowbreak(\frac{1}{2},\frac{1}{2})$.
\end{proposition}

\begin{remark}
By direct computation, just as in Propositions \ref{p1}-\ref{p2}, one can show
that the kernel of each homomorphism in Corollary \ref{hom} is $U_{k}$ for
some $k\in\left\{  -2,-1,0\right\}  $. Every such $U_{k}$ is irreducible (by
Corollary \ref{f1}). Therefore $V_{\alpha,\beta}^{\prime}\otimes
L(c_{p,q},h_{m,n})$ has a Jordan-H%
\"{o}%
lder composition of length 2.
\end{remark}

Based on the examples shown above and other examples not mentioned here (such
as $c_{2,7}=-68/7$ minimal models) we state

\begin{conjecture}
Let $c=c_{p,q}\neq0$ and let $L(c,h)$ be a minimal model. The module
$V_{\alpha,\beta}^{\prime}\otimes L(c,h)$ is reducible if and only if there
exists an admissible triple $((m,n),\allowbreak(m^{\prime},n^{\prime
}),\allowbreak(m^{\prime\prime},n^{\prime\prime}))$ such that $h=h_{m^{\prime
},n^{\prime}}$, $\alpha=h_{m,n}+h_{m^{\prime},n^{\prime}}-h_{m^{\prime\prime
},n^{\prime\prime}}$ and $\beta=1-h_{m,n}$. In this case there is an
irreducible submodule $U$ such that
\[
\left(  V_{\alpha,\beta}^{\prime}\otimes L(c,h)\right)  /U\cong
L(c,h_{m^{\prime\prime},n^{\prime\prime}}).
\]

\end{conjecture}

\subsection{c=1 intertwining operators\label{c=1}}

It is known (see \cite{Kac-Raina}, \cite{Frenkel-Zhu}) that $V(1,h)=L(1,h)$ if
and only if $h\neq\frac{m^{2}}{4}$ for $m\in\mathbb{Z}$. In case $h=m^{2}$,
the unique maximal submodule of $V(1,m^{2})$ is generated by a weight
$(m+1)^{2}$ vector and is isomorphic to $V(1,(m+1)^{2})$. We have the
following fusion rules for VOA $L(1,0)$ from \cite{Dong-Jiang}, \cite{Milas}:

Let $m,n,k\in\mathbb{N}$. An intertwining operator of type $\binom{L(1,k^{2}%
)}{L(1,m^{2})\text{\quad}L(1,n^{2})}$ exists if and only if $\left\vert
n-m\right\vert \leq k\leq n+m$. If $n\neq p^{2}$ for all $p\in\mathbb{Z}$,
then the operator of type $\binom{L(1,k)}{L(1,m^{2})\text{\quad}L(1,n)}$
exists if and only if $k=n$. Also, a transposed operator of type
$\binom{L(1,k)}{L(1,n)\text{\quad}L(1,m^{2})}$ exists if and only if $k=n$.

Now consider $L(1,1)$. For a fixed $m\in\mathbb{N}$ we have the operators of
types $\binom{L(1,(m-1)^{2})}{L(1,m^{2})\text{\quad}L(1,1)}$, $\binom
{L(1,m^{2})}{L(1,m^{2})\text{\quad}L(1,1)}$ and $\binom{L(1,(m+1)^{2}%
)}{L(1,m^{2})\text{\quad}L(1,1)}$ when $n\neq p^{2}$. Since $\alpha
=m^{2}+1-k^{2}\in\mathbb{Z}$, this proves reducibility of $V_{0,1-m^{2}%
}^{\prime}\otimes L(1,1)$. From the existence of a transposed operator of type
$\binom{L(1,n)}{L(1,n)\text{\quad}L(1,1)}$, we get reducibility of
$V_{0,1-n}^{\prime}\otimes L(1,1)$ for all $n\in\mathbb{N}$.

\begin{example}
Now we apply Theorem \ref{poly} to $V_{\alpha,\beta}^{\prime}\otimes L(1,1)$.
The Verma module $V(1,1)$ is reducible with degree 3, and the associated
polynomial has roots $-\alpha\pm2\sqrt{1-\beta}$ and $1-\alpha$ (see Example
\ref{s3}). If we set $\alpha=0$, then $U_{0}/U_{1}$ is the highest weight
module of the highest weight $1-\beta$. If $2\sqrt{1-\beta}\in\mathbb{Z}$, or
equivalently if $1-\beta=\frac{m^{2}}{4}$ for some $m\in\mathbb{Z}$, then
$U_{m-1}/U_{m}$ is the highest weight module of the weight $(\frac{m}%
{2}+1)^{2}$ and $U_{-m-1}/U_{-m}$ is the highest weight module of the weight
$(\frac{m}{2}-1)^{2}$. This indicates existence of the intertwining operators
of types
\[
\binom{L\left(  1,h\right)  }{L\left(  1,h\right)  \text{\quad}L\left(
1,1\right)  },\binom{L\left(  1,\left(  \frac{m+2}{2}\right)  ^{2}\right)
}{L\left(  1,\left(  \frac{m}{2}\right)  ^{2}\right)  \text{\quad}L\left(
1,1\right)  },\binom{L\left(  1,\left(  \frac{m-2}{2}\right)  ^{2}\right)
}{L\left(  1,\left(  \frac{m}{2}\right)  ^{2}\right)  \text{\quad}L\left(
1,1\right)  },
\]
for $h\in\mathbb{C}$, and $m\in\mathbb{Z}$ which agrees with (specially if
$h=\frac{m^{2}}{4}$) results from \cite{Dong-Jiang}.

Of course, for every $\alpha\in\mathbb{C\setminus Z}$ there are two nontrivial
subquotients, namely $U_{-\alpha\pm2\sqrt{1-\beta}-1}/U_{-\alpha\pm
2\sqrt{1-\beta}}$, but these are isomorphic to irreducible Verma modules with
highest weights $(\sqrt{1-\beta}\mp1)^{2}\notin\frac{1}{4}\mathbb{Z}^{2}$.
This could lead to the existence of the operators of type $\binom{L\left(
1,(\sqrt{h}-1)^{2}\right)  }{L\left(  1,h\right)  \text{\quad}L\left(
1,1\right)  }$ and $\binom{L\left(  1,(\sqrt{h}+1)^{2}\right)  }{L\left(
1,h\right)  \text{\quad}L\left(  1,1\right)  }$.
\end{example}

\begin{corollary}
If an intertwining operator of type $\binom{L\left(  1,h^{\prime}\right)
}{L\left(  1,h\right)  \text{\quad}L\left(  1,1\right)  }$ exists, then
$h^{\prime}\in\left\{  \left(  \sqrt{h}-1\right)  ^{2},h,\left(  \sqrt
{h}+1\right)  ^{2}\right\}  $. In particular if an intertwining operator of
type $\binom{L\left(  1,\left(  \frac{k}{2}\right)  ^{2}\right)  }{L\left(
1,\left(  \frac{m}{2}\right)  ^{2}\right)  \text{\quad}L\left(  1,1\right)  }$
exists for some $m,k\in\mathbb{Z}_{+}$, then $k\in\left\{  m-2,m,m+2\right\}
$.
\end{corollary}

\begin{proof}
Directly from Theorem \ref{red} and the previous example.
\end{proof}

\begin{example}
We can also apply Theorem \ref{poly} to $V_{\alpha,\beta}^{\prime}\otimes
L(1,\frac{1}{4})$. There is a singular vector in $V(1,\frac{1}{4})$ at level
2, so the roots of the associated polynomial are $-\alpha\pm\sqrt{1-\beta}$.
If we want $1-\beta=\frac{m^{2}}{4}$, then $\alpha=\frac{k}{2}$ for some
$k\in\mathbb{Z}$. In this case, $U_{\frac{\pm m-k}{2}-1}/U_{\frac{\pm m-k}{2}}
$ is the highest weight module of the weight $\frac{(m\mp1)^{2}}{4}$. This
gives the operators of type
\[
\binom{L\left(  1,(\frac{m-1}{2})^{2}\right)  }{L\left(  1,\left(  \frac{m}%
{2}\right)  ^{2}\right)  \text{\quad}L\left(  1,\frac{1}{4}\right)  }\text{
and }\binom{L\left(  1,(\frac{m+1}{2})^{2}\right)  }{L\left(  1,\left(
\frac{m}{2}\right)  ^{2}\right)  \text{\quad}L\left(  1,\frac{1}{4}\right)
}.
\]
Existence of these intertwining operators was proved in \cite{Milas}. For
$\alpha\notin\frac{1}{2}\mathbb{Z}$ we get a subquotient isomorphic to a
highest weight module of the weight $(\sqrt{1-\beta}\mp\frac{1}{2})^{2}$
indicating operators of type $\binom{L\left(  1,(\sqrt{h}-\frac{1}{2}%
)^{2}\right)  }{L\left(  1,h\right)  \text{\quad}L\left(  1,1\right)  }$ and
$\binom{L\left(  1,(\sqrt{h}+\frac{1}{2})^{2}\right)  }{L\left(  1,h\right)
\text{\quad}L\left(  1,1\right)  }$.
\end{example}

\begin{corollary}
If an intertwining operator of type $\binom{L\left(  1,h^{\prime}\right)
}{L\left(  1,h\right)  \text{\quad}L\left(  1,\frac{1}{4}\right)  }$ exists,
then $h^{\prime}=\left(  \sqrt{h}\pm\frac{1}{2}\right)  ^{2}$. In particular,
let $m,k\in\mathbb{Z}_{+}$. If an intertwining operator of type $\binom
{L\left(  1,\left(  \frac{k}{2}\right)  ^{2}\right)  }{L\left(  1,\left(
\frac{m}{2}\right)  ^{2}\right)  \text{\quad}L\left(  1,\frac{1}{4}\right)  }$
exists, then $k=m\pm1$.
\end{corollary}

\begin{remark}
Let $\varkappa\notin\mathbb{Q}$, $c=13-6\varkappa-6\varkappa^{-1}$ and
$\Delta(k)=\frac{k(k+2)}{4\varkappa}-\frac{k}{2}$ for some $k\in\mathbb{Z}$.
Then $V(c,\Delta(k))$ is irreducible if and only if $k<0$ (see
\cite{Feigin-Fuchs2}). Furthermore, if $k_{i}\in\mathbb{N}$, $i=1,2,3$ then an
intertwining operator of type $\binom{L\left(  c,\Delta(k_{3})\right)
}{L\left(  c,\Delta(k_{1})\right)  \text{\quad}L\left(  c,\Delta
(k_{2})\right)  }$ exists if and only if $k_{1}+k_{2}+k_{3}\in2\mathbb{Z}$ and
$\left\vert k_{1}-k_{2}\right\vert \leq k_{3}\leq k_{1}+k_{2}$ (see
Proposition 2.24. in \cite{Frenkel-Zhu 2}). However, if we allow $\varkappa
=1$, then $c=1$ and $\Delta(k)=\frac{k^{2}}{4}$. In this case it seems that
these results for the intertwining operators still hold.
\end{remark}

\end{document}